\documentclass{amsart}

\usepackage{url}

\usepackage{amscd}
\usepackage{amsfonts}
\usepackage{amssymb}
\usepackage{euscript}
\usepackage{amsmath}
\usepackage{amsthm}
\usepackage[matrix, arrow, curve]{xy}
\usepackage{mathrsfs}

\usepackage{dsfont}

\newcommand{\Li}{L_\infty}
\newcommand{\LL}{\mathop{L}}

\newcommand{\T}{T}
\newcommand{\B}{B}

\newcommand{\F}{F}
\newcommand{\R}{{{\mathbb R}^n}}

\newcommand{\Det}{\mathop{\mathrm{Det}}}
\newcommand{\sgn}{\mathrm{sgn}}
\newcommand{\C}{\mathscr{C}}
\newcommand{\g}{\mathfrak g}

\newcommand{\set}[1]{\underline{\mathbf{#1}}}
\newcommand{\We}{\mathop{\mathcal{W}^n}}

\newcommand{\FM}{\mathbf{FM}_{n}}
\newcommand{\ND}{N(Disk)}
\newcommand{\Dil}{\mathrm{Dil}(n)}
\newcommand{\fm}{\mathfrak{fm}_{n}}

\newcommand{\CS}[2]{\C(#1)(#2)}
\newcommand{\fCS}[2]{f\C(#1)(#2)}
\newcommand{\Cs}[2]{\C^{0}(#1)(#2)}

\newcommand{\Op}{\mathop{\mathcal{O}}}
\newcommand{\Set}{\mathbf{Set}}
\newcommand{\mul}{mul}
\newcommand{\gl}{SO(n)}
\newcommand{\fFM}{f\FM}
\newcommand{\ffm}{f\fm}
\newcommand{\Disk}{D}

\newcommand{\Mor}{\mathop{\mathrm{Mor}}}
\newcommand{\Hom}{\mathop{\mathrm{Hom}}}
\newcommand{\Ran}{Ran}
\newcommand{\ranp}{\varpi}
\newcommand{\col}{\mathfrak{Col}}
\renewcommand{\Col}{\mathfrak{col}}
\newcommand{\SetIso}{\Set_\simeq}
\newcommand{\SetInj}{\Set_{\hookrightarrow}}
\newcommand{\id}{\mathrm{id}}
\newcommand{\tens}[1]{\mathbin{\mathop{\otimes}\limits_{#1}}}
\newcommand{\ev}{\mathfrak{v}}
\newcommand{\EV}{\mathcal{V}}
\newcommand{\ealpha}{\mathrm{A}}
\newcommand{\Ealpha}{\mathcal{A}}

\renewcommand{\Im}{\mathop{\mathrm{Im}}}

\theoremstyle{plain}
\newtheorem{prop}{Proposition}
\newtheorem{theorem}{Theorem}

\theoremstyle{definition}
\newtheorem{definition}{Definition}
\newtheorem{example}{Example}

\theoremstyle{remark}

\begin{document}

\title{Weyl $n$-algebras}
\author{Nikita Markarian}

\date{}

\address{National Research University Higher School of Economics,
Department of Mathematics, 20 Myasnitskaya str., 101000, Moscow,
Russia}

\email{nikita.markarian@gmail.com}

\thanks{This research carried out in 2014-2015 was supported by “The National Research 
University Higher School of Economics’ Academic Fund Program” grant 14-01-0034.
}

\maketitle

\begin{abstract}
We introduce Weyl $n$-algebras and show how their factorization complex 
may be used to define invariants of manifolds. In the appendix, we 
heuristically explain why these invariants must be 
perturbative Chern--Simons invariants. 
\end{abstract}

\section*{Introduction}

The aim of this article is to develop the idea announced in \cite{M}:
Chern--Simons perturbative invariants of 3-manifolds introduced in 
\cite{AS, AS2, BC} may be defined by means of factorization complex
considered in \cite{BD, L,F, G}.  To get these invariants 
one has to calculate factorization homology of 
Weyl $n$-algebra, which is an object of independent interest.

An important property of Weyl $n$-algebras is that their factorization 
homology on a closed manifold is one-dimensional (Theorem \ref{one}).
It would be plausible to find some conceptual proof of this statement,
perhaps by using some kind of Morita invariance of factorization homology.
As far as I know, such arguments are unknown even in the classical situation,
when $n=1$. Besides, as I learned from O.~Gwilliam, in \cite{OG}
it is shown, that
the factorization algebra of any ``free'' BV theory has 
one-dimensional factorization homology over a closed manifold, which implies the result for the Weyl case.

Weyl $n$-algebras may be applied to the differential calculus in the sense 
of \cite{TT}.
For example, the $\Li$-morphism from the Lie algebra of polyvector
fields on a vector space, which is a Weyl $2$-algebra, to the
Lie algebra of endomorphisms of differential forms
on it (see e.~g. \cite{TT}) is given by the map analogous to 
the one from Proposition \ref{modules} for a 2-dimensional
cylinder. We hope to discuss this elsewhere.

I would like to draw the readers attention  to recent papers \cite{CPTVV} and \cite{GH},
which are closely related to  the present one.

In the first section, we shortly recall the definition of operad
and module over it, just to introduce notations. We send the reader to e.~g. \cite{L} 
for  a detailed treatment.

In the second section, we collect facts about Fulton--MacPherson operad
 and $\Li$ operad we need. 

Section \ref{fh} is devoted to the factorization complex.
There is nothing new here, this notion is deeply discussed in \cite{L}.
We use the Fulton-MacPherson compactification following  \cite{PS} and others.
In Subsection \ref{fhlah} a connection between Lie algebra
homology and the factorization complex is described. 
Proposition \ref{modules} interprets this connection
in terms of a morphism of right $\Li$-modules. The right
$\Li$-module $C_*(\fCS{M}{S})$
is similar to the Goodwillie derivative of the functor
$\Sigma^\infty\Hom_{\mathcal{T}op_*}(M,-)$
(see e.~g. \cite{AAC} and references therein).
This right $\Li$-module has an additional structure:
it is a pull back of a right $e_n$-module under the map
of operads $\Li\to e_n$ (compare with $\mathrm{KE}_L$-modules from \cite{AC}). 
It seems that invariants of manifolds we introduce below
reflect this additional structure.
 
In the Section  \ref{Wna} we introduce Weyl $n$-algebras.
The Euler structure on a
manifold, which we introduce in Subsection \ref{euler},
simplifies definition
of the factorization complex of a Weyl $n$-algebra on it.
I do not know, whether this is just a technical point, or
it has some deep relations with \cite{Tu}, where the term
is taken from.

As was already mentioned, factorization homology 
of a Weyl algebra on a closed manifold
is one-dimensional. It is easy to  produce a cycle 
presenting this the only class.
A more subtle and interesting question is to find a cocycle representing the class
dual to this cycle, which is an element of the dual complex.
For $n=1$, $M=S^1$ and generators of $\We$ of zero degrees
this question is solved in \cite{FFS}. 

If such a formula existed for any $n$ and $M$, it would
substantially simplify the last section, where we apply
Weyl $n$-algebras to the calculation of invariants of a manifold.
Instead of using the non-existent aforementioned formula we analyze
what happens with factorization homology when 
we collapse a homological sphere. 
The formula we get is similar to the one in \cite{AS} and \cite{BC}.

In Appendix I informally explain how our definition of invariants 
matches with the initial physical definition via path integral.

Methods and results of this paper may be naturally generalized for a manifold
with boundary or a pair of a closed 3-manifold with a link in it. I hope to elaborate
on these points in future papers.

\smallskip

{\bf Acknowledgments.} I am grateful to G.~Arone, P.~Bressler, D.~Calaque, A.~Cattaneo,
B.~Feigin, G.~Ginot, D.~Kaledin, A.~Kalugin and A.~Khoroshkin  for fruitful discussions. 
I deeply thank the referee for providing constructive comments and help in improving the contents of this paper.

\section{Operads}

\subsection{Definition}
Let $C$ be a symmetric monoidal category with product $\otimes$,
$\Set_{\hookrightarrow}$
be the category of finite sets and injective morphisms and $\SetIso$
be the category of finite sets and isomorphisms. 

A {\em unital operad}
$\Op$ in $C$ is defined by the following:
\begin{itemize}
\item A contravariant functor $\Op$ from the category of finite sets
and injective morphisms  $\Set_{\hookrightarrow}$ to $C$, the image
of the set of $k$ elements is called {\em operations of arity} $k$,
\item For any surjective morphism of sets $p\colon S\to S'$
a morphism called {\em composition} of operation is given
$$
\mul_p\colon \Op(S')\otimes \bigotimes_ {i\in S'}\Op(p^{-1}(i)) \to \Op(S) ,
$$
such that
\begin{itemize}
\item it is functorial with respect to injective morphisms $i\colon S_0\to S$
for which composition $p\circ i$ is surjective
\item  for any pair of surjective morphisms $S\stackrel{p}{\to}S'\stackrel{p'}{\to}S''$
equality
$$
\mul_{p' \circ p}=\mul_{p'}\circ\bigotimes_{i\in S'} \mul_{p^{-1} i} 
$$
holds.
\end{itemize}
\end{itemize}

A {\em non-unital operad} is defined by the same data, but with $\SetInj$
replaced by $\SetIso$, the category of finite sets and isomorphisms.

Any unital operad canonically produces a non-unital one by forgetting structure.

With any unital operad $\Op$ in $C$ one may associate the monoidal category 
$\Op^{\otimes}$
enriched over $C$. Its objects are labeled by finite sets.
A morphism between objects labeled by $S$ and $S'$
is given by a map $m\colon S\to S'$ of finite sets 
together with a collection 
$
\{\phi_i\in\Op(m^{-1}(i)) \,|\, i\in S'\}.
$ 
The composition of elements of $\Mor\nolimits_{\Op^{\otimes}} (S,S')$
given by surjective maps of sets is given by the 
composition of operations,
and composition with ones given by injective morphisms
is given by the action of $\Set_{\hookrightarrow}$.
For a non-unital operad the construction is the same, but the product
is taken only over surjective morphisms.

The operad may be reconstructed from the monoidal category $\Op^\otimes$ fibered 
over $\Set$.

A {\em colored operad} with a set of colors $B$ is a generalization of an operad. 
In the same way, it produces a monoidal category with objects labeled by $B^S$,
where $S$ runs over finite sets. So operations in a colored
operad are enumerated by finite set $\set{n}$ and a point in $B^{n+1}$ and composition
is a morphism from the fibered product over $B$. For details see e.~g. \cite[2.1.1]{L}.

We will consider operads fibered over the category of topological spaces, its definition is an obvious
modification of the previous one.  
\subsection{Modules}

\begin{definition}
For  $\Op$ an operad, a {\em left (right) $\Op$-module} in a category $M$ 
is a covariant (contravariant) functor from $\Op^{\otimes}$ to $M$.
\end{definition}

Let $D$ be a symmetric monoidal category with unit  $\mathds{1}$ 
and $\Op$ is an operad in 
a symmetric monoidal category $C$.
Given an object $A$ in $D$
and an element $e\colon \mathds{1}\to A$ there are natural functors
$\SetIso\to D$ and $\SetInj\to D$. The first one sends a set $S$ 
to $A^{\otimes S}$ with the natural action of isomorphisms of $S$.
The second one sends set $S$ to $A^{\otimes S}$ as well and
a morphism $S'\hookrightarrow S$ sends to
\begin{equation}
e^{\otimes (S\setminus S')}\otimes\id^{\otimes S}\colon A^{\otimes S'} \to A^{\otimes S}.
\end{equation}

\begin{definition}
{\em An algebra $A$ over a non-unital operad $\Op$} in $D$  is a left module over $\Op$
in $D$ such that its restriction  to $\SetIso$ is the functor as above.

{\em A algebra $A$ with unit $e\colon \mathds{1}\to A$
 over a unital operad $\Op$} in $D$ is a left module over $\Op$ 
in $D$ such that its restriction to $\SetInj$ is the functor as above.

Denote these modules by $A^{\otimes}$.
\end{definition}

Let $\Op$ be a dg-operad, that is an operad in the category of complexes.

\begin{definition}
Let $\Op$ be a dg-operad and $L$ and $R$ be a left and a right dg-modules over it.
Then the {\em tensor product} $L\otimes_{\Op}R$ of modules over the operad
is the tensor product of functors corresponding to modules from
$\Op^{\otimes}$ to the category of complexes.  
\label{tp}
\end{definition}

The definition works for both unital and non-unital and also colored operads. 
Given a unital operad $\Op$ denote by $\tilde \Op$ the corresponding
non-unital operad. The canonical  embedding 
$\tilde{\Op}^{\otimes}\hookrightarrow\Op^{\otimes}$
induces $\tilde{\Op}$-structure on any left and a right $\Op$-modules $L$ and $R$ 
the canonical map
\begin{equation}
L\otimes_{\tilde\Op}R\to L\otimes_{\Op}R.
\end{equation}

\section{Fulton-MacPherson operad}

\subsection{Fulton-MacPherson compactification}
\label{FMC}

Let $\R$ be an affine space. For a finite set $S$ let denote by $(\R)^S$ the set of ordered $S$-tuples in $\R$.
Let $\Cs{\R}{S}\subset (\R)^S$ 
be the configuration space of distinct ordered points in $\R$ labeled by $S$.
In \cite{GJ, Ma} (see also \cite{PS} and \cite{AS}) the Fulton--MacPherson 
compactification 
$\CS{\R}{S}$ of $\Cs{\R}{S}$ is introduced.
This is a manifold with corners and a boundary with interior
$\imath\colon \Cs{\R}{S}\hookrightarrow \CS{\R}{S} $.
There is a projection $\pi\colon \CS{\R}{S}\to (\R)^S$
such that $\pi\circ\imath\colon\Cs{\R}{S} \to (\R)^S $ is the natural embedding.

For any $S'\subset S$ there is the projection map
\begin{equation}
 \CS{\R}{S} \to \CS{\R}{S'},
\end{equation}
compatible with the same maps $\Cs{\R}{S} \to \Cs{\R}{S'}$ and 
$(\R)^S\to (\R)^{S'}$.

The natural action of the group of affine transformations on $\Cs{\R}{S}$ is lifted to
$\CS{\R}{S}$. Denote by $\Dil$ its subgroup consisting of dilatations and shifts. 
Group $\Dil$ acts freely on $\CS{\R}{S}$ and the quotient is isomorphic to
the fiber $\pi^{-1}(\vec{0})$, where $\vec{0}\in(\R)^S$ is the $S$-tuple
sitting at the origin.
To build this isomorphism consider dilatations with positive coefficients
with the center at the origin: $\mathbb R_{>0}\times \Cs{\R}{S}\to \Cs{\R}{S}$.
By the construction of the compactification their action is lifted to  
$r\colon\mathbb R_{\ge 0}\times \CS{\R}{S}\to \CS{\R}{S}$, which is a fiber bundle.
The map $r(0\times -)$ factors through the quotient by $\Dil$ 
and its image  lies in $\pi^{-1}(\vec{0})$.
This gives the required isomorphism.
It follows that $\pi^{-1}(\vec{0})$ is a retract of $\CS{\R}{S}$.

As it is just mentioned, manifolds with corners $\CS{\R}{S}/\Dil$ and $\pi^{-1}(\vec{0})$ are isomorphic.
Denote any of these manifolds by $\FM^{S}$.
The sequence of manifolds $\FM^{S}$ is a contravariant functor from $\Set_{\hookrightarrow}$
to topological spaces: the map corresponding to an embedding of sets
forgets points that are not in its image.
The sequence $\FM^{S}$ may be equipped with a structure
of a unital operad in the category of topological spaces. This operad is a free
as an operad of sets and as such is generated by quotients of $\Cs{\R}{S}\hookrightarrow\CS{\R}{S}$
by $\Dil$.
The action of $k$-ary operations $\Cs{\R}{\set{k}}/\Dil$ on $\CS{\R}{S}$
looks as follows. Consider the submanifold  of $\CS{\R}{S}$ for which
the image of $\pi\colon \CS{\R}{S}\to (\R)^S $ consists exactly of $k$ different points. 
This submanifold is isomorphic to $\Cs{\R}{\set{k}}\times \pi^{-1}(\vec{0})$
because fibers of $\pi$ over any point are isomorphic due to parallel translations.
The embedding of this submanifold to $\CS{\R}{S}$ in composition with
the quotient by $\Dil$ gives a map
$$
\CS{\R}{\set{k}}/\Dil\times {(\FM)}^{\times k}\to \CS{\R}{\bullet}/\Dil=\FM,
\label{assembly}
$$
which is the desired action, where $\set{k}$ is the set of $k$ elements.

\begin{definition}
The sequence of topological spaces $\FM^{S}$ with the unital operad structure as above
is called the {\em Fulton--MacPherson operad}.
\end{definition}

\subsection{Chains of Fulton-MacPherson operad}

Given a topological operad, one may produce a dg-operad by
taking complexes of chains of its components.

\begin{definition}
Denote by $\fm$ the operad of $\mathbb R$-chains of $\FM$. 
\label{chains}
\end{definition}

Real numbers appear here are to simplify things, in fact all object
and morphism we shall use may be defined over rationals, see remark before Example \ref{moyal} 
below.

By chains we mean the complex of de Rham currents, that is why we need real chains.
Alternatively, one may think about the cooperad of de Rham cochains of $\FM$.

\begin{prop}
Operad $\fm$ is weakly homotopy equivalent to $e_n$,
the operad of chains of the little discs operad.
\label{eq}
\end{prop}

\begin{proof}
See \cite[Proposition 3.9]{PS} and Subsection \ref{disk} below.
\end{proof}

Spaces $\FM^{S}$ are acted on by the general linear group, and, in particular,
by its maximal compact subgroup $\gl$, we suppose that a scalar
product on the space is chosen. The semidirect product $\FM\rtimes\gl$ 
is called the operad of framed disks $\fFM$. Any operad is equipped with a natural
structure of an operad colored over the classifying space of its invertible 1-ary elements.
In this way, we will consider 
$\fFM$ as an operad colored by the classifying space $B\gl$.

\begin{definition}
Denote by $\ffm$ the operad of $\mathbb R$-chains of $\fFM$.
\label{ffm} 
\end{definition}

The closely connected, but not identical object is the operad of framed disks
from \cite{Get}. And much like with Definition \ref{chains}, real numbers may be replaced with rational for our purposes.

Operations of arity $s$ of $\ffm$ form complexes over  $B\gl^{s+1}$. 
An algebra over $\ffm$ is given by a family of complexes over appropriate powers of $B\gl$.
Below we will need only the following restrictive, but a simpler class of such algebras.

\begin{definition}
We say that a dg-algebra $A$ over $\fm$ is {\em invariant},
if all structure maps of complexes
$$
\fm \otimes A\otimes\cdots\otimes A\to A
$$
are invariant under the action of group $\gl$ on complexes of operations
of $\fm$. An invariant algebra over $\fm$ is naturally an algebra over $\ffm$. 
\label{inv}
\end{definition}

Note, that we mean invariance on the level of complexes, not up to homotopy.
An important class (and the only class we need, in fact) of invariant $e_n$-algebras is 
universal 
enveloping $e_n$-algebras, see the end of the next Subsection.

\subsection{$\Li$ operad}
\label{li}

A {\em tree} is an oriented connected graph with three type of vertices: the
{\em root} has one incoming edge and no outgoing ones, {\em leaves}
have one outgoing edge and no incoming  ones and {\em internal
vertexes} have one outgoing edge and more than one incoming ones.
Edges incident to leaves will be called {\em inputs}, the edge incident
to the root will be called the {\em output} and all other edges will be called {\em internal
edges}. The degenerate tree has one edge and no internal vertexes.
Denote by $\T_k(S)$ the set of non-degenerate trees with $k$ internal
edges and leaves labeled by a finite set $S$.

For two trees $t_1\in T_{k_1}(S_1)$ and $t_2\in T_{k_2}(S_2)$ and an element $s\in S_1$
the composition of trees $t_1\circ_s t_2\in T_{k_1+k_2+1}$ is obtained by identification of 
the input of $t_1$
corresponding to $s$ and the output of $t_2$. Composition of trees is associative and
the degenerate tree is the unit.  The set of trees with respect to the composition forms an 
operad.

We call a tree with only one internal vertex the {\em star}.
Any  non-degenerate tree with $k$ internal edges may be uniquely presented as a composition of
$k+1$ stars.

The operation of {\em edge splitting} is the following: take
a non-degenerate tree, present it as a composition of stars and
replace one star  with a tree that is  a product of two stars and
has the same set of inputs. The operation of  an edge splitting
depends on an internal vertex and a proper subset of incoming edges with 
more than one element.

For a non-degenerate tree $t$ denote by $\Det(t)$ the
one-dimensional $\mathbb Q$-vector space that is the determinant of
the vector space generated by internal edges. For $s>1$ consider the
complex
\begin{equation}
L(s)\colon \bigoplus_{t\in T_0(\set{s})} \Det (t) \to
\bigoplus_{t\in T_1(\set{s})} \Det(t) \to \bigoplus_{t\in
T_2(\set{s})} \Det (t) \to \cdots,
 \label{complex}
\end{equation}
where $\set{s}$ is the set of $s$ elements,
the cohomological degree of a tree $t\in T_k(\set{s})$ is $2-s+k$ and the  
differential is given by all possible splittings  of an edge (see e.~g.
\cite{GK}). The composition of trees equips the sequence
$L(i)\otimes \sgn$ with the structure of a non-unital $dg$-operad, here $\sgn$ is the
sign representation of the symmetric group.

This operad is called the {\em $\Li$ operad}.
For simplicity denote by the same symbol the operad $\Li\otimes_{\mathbb Q}\mathbb{R}$,
it will be clear from the context which one is meant.
Denote by $\Li[n]$ the
$dg$-operad given by the complex $L(s)[n(s-1)]\otimes(\sgn)^n$
and refer to it as $n$-shifted $L_\infty$ operad.

As $\FM$ is freely generated  by $\Cs{\R}{S}/\Dil$
as the operad of sets, there is a map $\mu$ from it to the free
operad with one generator in each arity, which sends  generators to generators.
Elements of the latter operad are enumerated by rooted trees. The map above sends
$\C_{\set{k}}^0(\R)/\Dil$ to the star tree with $k$ leaves.
For a tree $t\in\T(S)$ denote by 
$[\mu^{-1} (t)]\in C_*(\F_n(S))$ the chain presented by its preimage under $\mu$.

\begin{prop}
Map $[\mu^{-1} (\cdot)]$ as above gives a morphism 
from shifted $L_\infty$ operad $L(s)[s(1-n)]$  
to the $dg$-operad $\fm$ of chains 
of the Fulton--MacPherson operad.
The last operad here is treated as a non-unital one. 
\label{tree}
\end{prop}

\begin{proof}
To see that the map commutes with the differential, note, that two strata given by $\mu$ with 
dimensions differing by 1
are incident if and only if  one of the corresponding trees is obtained from another by  edge splitting.
In this way, we get a basis in the conormal bundle to a stratum labeled by the internal edges,
it follows that orientations on the chains of the boundary of a stratum match
 signs in the complex (\ref{complex}). 
\end{proof}
It follows that there is a morphism of $dg$-operads
\begin{equation}
\Li[1-n]\to \fm
\label{morphism}
\end{equation}

\begin{definition}
For a $\fm$-algebra $A$ call its pull-back under (\ref{morphism})
the {\em associated $\Li$-algebra} and denote it by $\LL(A)$.
\label{la}
\end{definition}

Since the operad $\fm$ is weakly homotopy equivalent to $e_n$
(Proposition \ref{eq}),
it gives a homotopy morphism of operads $\Li[1-n]\to e_n$.

This morphism of operads produces a functor from the category
of $e_n$-algebras to that of $\Li$-algebras. This functor has a left adjoint,
which is called the universal enveloping $e_n$-algebra.
The important example of the latter is the complex of rational 
chains of an iterated loop space $\Omega^n X$, which is a universal enveloping 
$e_n$-algebra of the homotopy groups Lie algebra $\pi_{*-1}(X)$,
for more details see e.~g. \cite[Section 5]{F}.
Note, that $\Omega^n X$ is equipped with a natural $SO(n)$ action.
This is in good agreement with the fact that any universal enveloping
$e_n$-algebra is invariant.

\section{Factorization homology}
\label{fh}
\subsection{Factorization complex}
\label{FH}

Let  $M$ be a $n$-dimensional oriented topological manifold.
In the same way, as for $\R$  there is the
Fulton--MacPherson compactification
$\CS{M}{S}$ of the space $\Cs{M}{S}$ of ordered pairwise distinct points in $M$
labeled by $S$. Locally it is the same thing.
Inclusion $\Cs{M}{S}\hookrightarrow \CS{M}{S}$ is a homotopy equivalence,
there is  a projection $\CS{M}{S} \stackrel{\pi}{\to} M^S$.

Recall that a point in the Fulton--MacPherson compactification $\CS{\R}{S}$ of 
the configuration space of $\R$
looks like a configuration from the configuration space $\Cs{\R}{S'}$
with elements of $\FM$ sitting at each point of the configuration.
It follows that spaces $\CS{\R}{\bullet}$ form a right $\FM$-module, 
which  is freely generated by $\Cs{\R}{\bullet}$ as a set.
The same is nearly true for the Fulton--MacPherson compactification
of any oriented manifold $M$. But to define such an action one needs to choose
coordinates at the tangent space of any point of a configuration of $\CS{M}{S}$.
To fix it one has to consider either only framed
manifolds or introduce framed configuration space.

\begin{definition}
The framed Fulton--MacPherson compactification $\fCS{M}{S}$ is the principal
$\gl^S$ bundle over $\CS{M}{S}$, which is the pullback of product of principal
bundles associated with the tangent bundles to each point under the projection map
$\pi \colon \CS{M}{S}\to M^{S}$.
\label{ffh}
\end{definition}

The chain complex  $C_*(\fCS{M}{S})$
over $B\gl^{S}$ for various $S$ make up a right $\ffm$-module 
(see Definition \ref{ffm}).

\begin{definition}
For an algebra $A$ over $\ffm$  and an oriented manifold $M$ the {\em factorization complex}
$\int_M A$ is the tensor product (Definition \ref{tp}) of the left $\ffm$-module
$A^\otimes$ and the right $\ffm$-module $C_*(\fCS{M}{S})$.
\label{definition}
\end{definition}

The homology of $\int_M A$ is called the {\em factorization homology} of $A$ on $M$.

For an invariant $\fm$-algebra (Definition \ref{inv}) the definition of the factorization complex may be rephrased as follows.

\begin{prop}
For an invariant unital $\fm$-algebra $A$ and an oriented manifold $M$ the
factorization complex $\int_M A$ is
the complex given by the colimit of the diagram
\begin{equation}
\begin{CD}
\bigoplus\limits_{S'} C_*(\CS{M}{S'})\tens{Aut(S')} A^{\otimes S'}\\
@AAA \\
\bigoplus\limits_{i\colon S'\to S}C_*(f\Cs{M}{S})\tens{\gl^S\rtimes Aut(S)}\bigotimes\limits_{s\in 
S} 
(\fm(i^{-1}s)\tens{Aut(i^{-1}s)} A^{\otimes 
{(i^{-1}s)}}) \\
@VVV \\
\bigoplus\limits_{S}C_*(\Cs{M}{S})\tens{Aut(S)} A^{\otimes S} 
\end{CD}
\label{coend}
\end{equation} 
where the summation in the middle runs over maps between finite sets,
the downwards arrow is given by the left action of $\fm$ on A for
$\Im i$ and the unit for $S\setminus \Im i$
and the upwards arrow is given by the right action of $\fm$ on 
$C_*(f\CS{M}{\bullet})$.
\label{colimit}
\end{prop}

\begin{proof}
The formula is a direct interpretation of Definition \ref{definition}.
\end{proof}

If the manifold is framed, that is its tangent bundle is trivialized,
the definition may be simplified: one should substitute  $\Cs{M}{S}$
instead of $f\Cs{M}{S}$ and remove  $\gl$ from the tensor product.

Since the upwards arrow in (\ref{coend}) is an isomorphism of underlying vector spaces, for any  class of the colimit above 
there is a unique chain downstairs, which is in the interior
of the Fulton--MacPherson compactification, that is in a configuration
space of distinct points. Thus on the complex (\ref{coend}) 
(that calculates the factorization homology) there is
an increasing filtration by the number of points of the configuration
space and the associated graded object is $\bigoplus_{S} C_*(\Cs{M}{S})\otimes A^{\otimes S}$.

Note, that this filtration splits as a filtration of vector spaces.
Thus, any morphism from or to the factorization complex
may be presented as the one for all graded  pieces of the filtration
consistent in a proper way.

The definition above may be again rephrased as follows. 
Denote by $\Ran(M)$ the {\em Ran space} 
of $M$,
that is the set of finite subsets of $M$ with the natural topology.
There is a natural map $M^{\times i}\to \Ran(M)$, which sends a set of points to its support.
Denote the composite map $\CS{M}{\set{i}}\to M^{\times i}\to \Ran(M)$ by $\ranp_i$.
The fiber of this map is the product of some copies
of the Fulton--MacPherson operad. Take a $\fm$-algebra $A$ and consider
chains $\bigoplus_i C_*(\CS{M}{\set{i}})\otimes_{\Sigma_i} A^{\otimes i})$ modulo 
relations (\ref{coend}). As all
relations respect $\ranp_*$, for any open subset of
the complex of these chains modulo relations is defined;
being restricted $\Cs{M}{\set{i}}\hookrightarrow \Ran(M)$ this complex
equals to $C_*(\Cs{M}{\set{i}})\otimes_{\Sigma_i}\otimes A^{\otimes i}$.
The way these complexes are glued together
defines a cosheaf (see e.~g. \cite{Cur}) on the Ran space. The factorization
homology is homology of this cosheaf, for details see \cite{L}.

\subsection{Polynomial algebra}
\label{comm}
Any commutative algebra 
canonically is an algebra over chains of any topological operad, because it is 
the operad of chains of the terminal object in the category of topological operad.
In particular, any commutative algebra is an $\fm$-algebra over and it is invariant.
 
Let $A$ be the polynomial algebra $k[V]$ generated by a $\mathbb{Z}$-graded vector space $V$
over the base field $k$ of characteristic zero containing $\mathbb R$. Its factorization complex
$\int_M A$ is a 
commutative algebra because 
any commutative algebra is a commutative algebra in the category
of commutative algebras. The multiplication in $\int_M A$ is given by taking unions
of points in $M$ and multiplication of labels for coinciding points.  

\begin{prop}[see { \cite[Ch. 4.6]{BD}}, \cite{GTZ}]
$\int_M A=k[H_*(M)\otimes V]$, where $H_*(M)$ is the integer homology
 groups of $M$ negatively graded. 
\label{com}
\end{prop} 

\begin{proof}
Choose a homogeneous basis of $V$
enumerated by a set $B$. The action of $\fm$ on a commutative algebra factorizes
through the augmentation map $\fm(\bullet)\to k$. 
It means, that the complex $\oplus A^{\otimes i}\otimes C_*(\CS{M}{\set{i}})$
modulo relations (\ref{coend}) 
equals to $\oplus A^{\otimes i}\otimes \overline{C}_*(\CS{M}{\set{i}})/ \sim$,
where $\sim$
are relations given by the unit and $\overline{C}_*(\CS{M}{\set{i}})$ is 
the chain complex of the Fulton--MacPherson compactification
with all border components shrunk to points. The latter space is 
simply the power $M^{\times i}$. 
Thus taking into account
relations $\sim$ we see that $\int_M A$ is the homology
of space of finite subsets of $M$ labeled by $B$,
that is the direct sum of homology of 
$M^{\times i_1}\times \cdots \times M^{\times i_{|B|}}$
modulo the action of product of symmetric groups
$\Sigma_{i_1}\times \cdots \times \Sigma_{i_{|B|}}$,
which is given by permutations for components that corresponds
to elements of the basis of even degree and by permutation
multiplied by the sign representation for odd degrees.
The multiplication on this space is obviously defined.
\end{proof}

\subsection{Disk operad}
\label{disk}

In this Subsection, we sketch a connection between 
our definition (which follows \cite{PS} and others) 
of the factorization homology and the one given in \cite{L, G, F}.

Given a $\fm$-algebra $A$ let us calculate its factorization homology 
 on the disk $\Disk=\{x\in\R|\lvert x\rvert<1\}$.

\begin{prop}
For a $\fm$-algebra $A$ the factorization  complex $\int_{\Disk} A$ is homotopy equivalent to 
$A$.
\label{ret}
\end{prop}

\begin{proof}
Define a morphism $A\to \int_D A$ as $a \mapsto [O]\otimes a$,
where $a\in A$ and $[O]$ is the $0$-cycle
presented by the origin of coordinates. To define the morphism in 
the opposite direction recall, that operations of the Fulton--MacPherson
operad are given by quotients $\CS{\R}{S}/\Dil$. Define morphism from the factorization complex 
$\int_D A$ to $A$ as the composite map
$$
C_*(\CS{\R}{S})\otimes A^{\otimes S}\to C_*(\CS{\R}{S}/\Dil)\otimes A^{\otimes 
S}=\fm(S)\otimes A^{\otimes S}\to A, 
$$
where the first arrow is given by the projection and the last arrow is the action
of operad. We have to show that  composition of this map with the previous
one is homotopic  to the identity map.
To build the homotopy  consider a retraction of the disk to the origin of coordinates.
Arguments as in the beginning of Subsection \ref{FMC} shows that it induces the homotopy we 
need.
\end{proof}

Embedding of disks into a bigger disk induces a map
from tensor powers of $\int_D A$ to $\int_D A$ itself
parametrized by the space of disks embedding.
This produces  an action on $\int_{\Disk} A$
of the nerve of disks operad  $\ND$ in the sense of \cite{L}, which is homotopy equivalent to $e_n$.
Moreover, the definition \cite[Definition 5.3.2.6]{L} of factorization homology
$\ND$-algebra being applied to $\int_{\Disk} A$ 
gives the same result as the definition we use for
factorization homology of $A$.

\subsection{Factorization homology and Lie algebra homology}
\label{fhlah}

Following the definition of a tree from the beginning
of Subsection \ref{li}, we say that a {\em bush} is
an oriented connected graph with three type of vertices:
{\em root} has no outgoing ones, {\em leaves}
have one outgoing edge and no incoming  ones and 
{\em internal vertexes} have one outgoing edge and more than one 
incoming ones.  
That is the only difference is that the root may have many incoming edges.
The composition of bushes is not defined,
but one may compose a tree and a bush by identification of 
an input of the bush and the output of the tree.
Thus, bushes form a right module over the operad of trees.
Denote by $\B_k(S)$ the set of bushes with $k$ 
edges not incident to leaves and leaves labeled by a set $S$.

Continuing on the same lines, define 
the operation of {\em edge splitting} in the same way as for trees:
we choose a vertex and a subset of incoming edges with more than one element,
then we cut off trees that grow from the chosen edges,
then glue an incoming edge to the vertex we choose and then glue
trees we cut to the input of the glued edge.
Note that an edge splitting for a bush may be done not only for
an internal vertex, but for a root as well.
But for an internal edge, the subset of edges must be
proper and for the root it may be the whole set.

For a bush $b$ denote by $\Det(b)$ the
one-dimensional $\mathbb Q$-vector space that is the determinant of
the vector space generated by internal edges. For $s>0$ consider the
complex
\begin{equation}
B(s)\colon \bigoplus_{b\in B_0(\set{s})} \Det (b) \to
\bigoplus_{b\in B_1(\set{s})} \Det(b) \to \bigoplus_{b\in
B_2(\set{s})} \Det (b) \to \cdots,
\label{bush-complex}
\end{equation}
where $\set{s}$ is the set of $s$ elements,
the cohomological degree of a bush $B\in B_k(\set{s})$ is $k-s$ and the  
differential is given by all possible splitting  of an edge. 
The composition of a tree and a bush
is compatible with differentials on complexes (\ref{complex}) and (\ref{bush-complex}) and thus
equips the complex with a structure of right
$\Li$-module.

Given 
a $\Li$-algebra $\g$ its homology (with trivial coefficients)
may be calculated by means of the homological Chevalley--Eilenberg
complex.
Its $n$-th term is the symmetric power $S^n(\g[1])$
and the differential is the coderivation defined
by the operations $l_i\colon S^i(\g[1])\to \g[1]$
corresponding to star trees (for the definition of the latter see Subsection \ref{li}).

This definition 
may be nicely formulated in terms of modules over operads as follows.

\begin{prop}
For a $\Li$-algebra $\g$ the product $\g^\otimes\otimes_{\Li} B(\bullet)$
 is isomorphic to the Chevalley--Eilenberg complex calculating homology
of $\g$ with trivial coefficients modulo the zero-degree component.  
\label{trivial}
\end{prop}

\begin{proof}
The proof is straightforward. For a more conceptual treatment see \cite{Bal}.
\end{proof}

The homology of a $\Li$-algebra with coefficients in the adjoint module
is calculated by the complex with $n$-th term  $S^n(\g[1])\otimes \g$.
The differential is a sum of the Chevalley--Eilenberg differential
and the coderivation $d_{ad}\colon \g\otimes S^n(\g[1])\to \bigoplus_i  S^i(\g[1])$ 
given by the adjoint action.
A light modification of the foregoing allows us 
to define it in terms of modules over operads.

A {\em marked bush} is a bush  with one of the edges
incoming to root marked.
Denote by $\B'_k(S)$ the set of marked bushes with $k$ non-marked
edges not incidental to leaves and leaves labeled by a set $S$. The edge splitting
for marked bushes is defined in the same way, if the root vertex
is chosen then the inserted edge is marked if the chosen subset
of edges contains the marked edge and is not marked otherwise.

As before, for a bush $b$ denote by $\Det(b)$ the
one-dimensional $\mathbb Q$-vector space that is the determinant of
the vector space generated by not marked edges. For $s>0$ consider the
complex
\begin{equation}
B'(s)\colon \bigoplus_{b\in B'_0(\set{s})} \Det (b) \to
\bigoplus_{b\in B'_1(\set{s})} \Det(b) \to \bigoplus_{b\in
B'_2(\set{s})} \Det (b) \to \cdots,
\label{bush'-complex}
\end{equation}
where $\set{s}$ is the set of $s$ elements,
the cohomological degree of a bush $B\in B_k(\set{s})$ is $k-s$ and the  
differential is given by all possible splitting  of an edge. 
The composition of a tree and a bush again 
equips the complex with a structure of right
$\Li$-module.

On the analogy of Proposition \ref{trivial} we have the following.

\begin{prop}
For a $\Li$-algebra $\g$ the product $\g^{\otimes}\otimes_{\Li} B'(\bullet)$
is isomorphic to the Chevalley--Eilenberg complex calculating homology
of $\g$ in the adjoint module.  
\end{prop}

\begin{proof}
The proof is straightforward. For a more conceptual treatment see \cite{Bal}.
\end{proof}

In Subsection \ref{li} we have defined a morphism from operad $\Li$ to $\fm$.
Applying this morphism to the right $\fm$-module
$C_*(\fCS{M}{S})$ introduced in Subsection \ref{FH}
we get the right action of $\Li$ on $C_*(\fCS{M}{S})$.
A morphism from the right $\Li$-module given by complexes (\ref{bush-complex}) and 
(\ref{bush'-complex})
generated by bushes to this right $\Li$-module 
produces morphisms from Chevalley--Eilenberg complexes
to the factorization complex. It may be formulated as follows.

\begin{prop}
Let $A$ be an invariant $\fm$-algebra. Let 
$C_{Ch}=(S^*(\LL(A)[1]), d_{Ch})$ 
$C_{Ch}^{ad}=( S^*(\LL(A)[1])\otimes \LL(A), d_{Ch}+d_{ad})$
be the Chevalley--Eilenberg
complexes calculating the homology of $\Li$-algebra $\LL(A)$ with
trivial coefficients and
in the adjoint module correspondingly. Let $M$ be a closed manifold and $p\in M$
is a point. Then morphisms
\begin{eqnarray*}
a_1\otimes\cdots \otimes a_i&\mapsto& [\Cs{M}{\set{i}}]\otimes_{\Sigma_i} (a_1\otimes\cdots 
\otimes a_i) \\
a_1\otimes\cdots \otimes a_i\otimes a_0&\mapsto& [\Cs{M\setminus p}{\set{i}}] 
\otimes_{\Sigma_i} (a_1\otimes\cdots 
\otimes a_i)\otimes a_0
\end{eqnarray*}
define maps from complexes $C_{Ch}(\LL(A))$ and $C_{Ch}^{ad}(\LL(A))$ 
respectively to the factorization complex  $\int_M 
A$, where $[\Cs{M}{S}]$, $[\Cs{M\setminus p}{S}]$ and $[p]$ are cycles in
$C_*(\CS{M}{S})$ presented by the configuration space of distinct points,
distinct points different from $p$ and the point $p$.
\label{modules}
\end{prop}

\begin{proof}
These morphisms are given by morphisms of right modules 
over $\fm$,
see the discussion before the Proposition.
\end{proof}

The first map above was introduced in  \cite{M} in a more explicit form.  

\section{Weyl $n$-algebra}
\label{Wna}

\subsection{Definition}
\label{v}

The usual Weyl algebra is a deformation of the polynomial
algebra. We have seen that a commutative algebra is  an algebra
over operad $\fm$  for any $n$. The analogous deformation of
a commutative algebra in the category of $\fm$-algebras 
gives us what we call the Weyl $n$-algebra.

Let $V$ be a $\mathbb{Z}$-graded  finite-dimensional vector space  over the base field $k$
of characteristic zero containing $\mathbb R$ equipped with a non-degenerate  skew-symmetric pairing 
$\omega\colon V\otimes V\to k$
of degree $1-n$. Let $k[V]$ be the polynomial algebra generated by $V$
and $k[[h]]$ be the ring of formal series and $k[V][[h]]$ is the polynomial algebra over it.
Denote by 
\begin{equation}
\partial_\omega \colon k[V]\otimes k[V] \to k[V]\otimes k[V]
\label{omega}
\end{equation}
the differential operator that is a derivation in each factor
and acts on generators as $\omega$. 

Consider $\FM(\set{2})$,  the space of 2-ary operations 
of the Fulton--MacPherson operad. This is  homeomorphic to the $(n-1)$-dimensional
sphere. Denote by $v$ the standard $\gl$-invariant 
$(n-1)$-differential form on it.  For any two-element subset $\{i,j\}\subset S$ 
denote by $p_{ij}\colon \FM(S)\to \FM(\set{2})$ the map that forgets
all points except ones marked by $i$ and by $j$. 
Denote by $v_{ij}$ the pullback of $v$ under projection $p_{ij}$.
Let $\alpha$ be an element of endomorphisms of
$k[V]^{\otimes S} \tens{Aut (S)}C^*(\FM(S))$ (where $C^*(-)$ is the de Rham complex)
given by 
$$
\alpha= \sum_{i,j\in S}\partial_\omega^{ij}\wedge v_{ij} ,
$$
where $\partial_\omega^{ij}$ is the operator $\partial_\omega$
applied to the $i$-th and $j$-th factors.
 
\begin{prop}
The composition
$$
k[V]^{\otimes S}\stackrel{\exp(h\alpha)}{\longrightarrow}k[V][[h]]^{\otimes S}\otimes C^*(\FM(S))\stackrel{\mu}{\to}k[V][[h]]\otimes C^*(\FM(S)),
$$
where $\mu$ is the product in the polynomial algebra,
defines a $k[[h]]$-algebra over the operad $\fm$ with the underlying space
$k[V][[h]]$.
\label{prod}
\end{prop}

\begin{proof}
This is a simple check. 
\end{proof}

The algebra defined in this way is obviously invariant under the action of $\gl$,
thus it is invariant (see Definition \ref{inv}).

\begin{definition}
For a pair $(V, \omega)$ as above the invariant $\fm$-algebra given
by Proposition \ref{prod} is called the {\em Weyl $\fm$-algebra}. Denote it by $\We_h(V)$.
\label{wna}
\end{definition}

Note that Proposition \ref{eq} provides us with the {\em Weyl $e_n$-algebra}.

One may give an alternative definition of the Weyl algebra as the  universal 
enveloping of the Heisenberg Lie algebra,
compare with \cite[3.8.1]{BD}. 
It allows us to define the rational version of the Weyl algebra, which is an  algebra 
over rational
chains of the Fulton--MacPheson operad.

\begin{example} For $n=1$ and a vector space of degree $0$
one gets the Moyal product.
\label{moyal}
\end{example}

Denote by $\We(V)$ the algebra over Laurent formal series, which is 
 the localization  $\We_h(V)\hat{\otimes}_{k[h]} k[h^{-1}, h]]$.
Both of algebras $\We_h(V)$ and $\We(V)$ are equipped
with increasing filtration, which is multiplicative
with respect to the commutative product on $k[V][[h]]$,  and $h$ and elements of $V$ lie in the component of degree 1.

Consider the $\Li$-algebra  $\LL(\We_h(V))$ associated with
the Weyl algebra . By the very definition,
all operations on it are given by integration of closed forms by chains of the Fulton--MacPherson
operad. But one may see, that chains representing higher operations
(that is operations, which are not composition of Lie brackets) in $\Li$
are all homologous to zero, because $\Li$ is a resolution of the Lie operad. 
Thus $\LL(\We_h(V))$ is a $\mathbb{Z}$-graded Lie algebra, all higher operations vanish.
This Lie algebra $\LL(\We_h(V))$ is a deformation of the Abelian one.
The first order deformation gives the {\em Poisson Lie algebra}: 
the underlying space is the $\mathbb{Z}$-graded commutative algebra
$k[V][[h]]$, the bracket is defined by $h\omega\colon V\otimes V\to k[[h]]$ 
on generators and
satisfies the Leibniz rule.
For the classical one-dimensional Weyl  algebra it is known, that 
higher terms of the deformation are non-trivial:
$\LL(\mathcal{W}^1_h(V))$
differs from the Poisson Lie algebra (\cite{V}).
But for $n>1$ the situation is  simpler.

\begin{prop}
For $n>1$
Lie algebra $\LL(\We)$ is isomorphic to the Poisson Lie algebra
of $(V\hat\otimes k[h^{-1},h]], \omega)$ over $k[h^{-1},h]]$, the definition of the latter is as 
above.
\label{poiss}
\end{prop}

\begin{proof}
Obvious, because for $n>1$ the square of the de Rham cochain $v$ is zero.
\end{proof}

\subsection{Factorization homology of $\We$}
\label{FHoW}

Weyl $n$-algebra is a deformation of a commutative algebra.
From Subsection \ref{comm} we know factorization homology
of a commutative algebra. Below we use deformation arguments  
to calculate factorization homology
of the Weyl algebra on a closed manifold $M$.

\begin{theorem}
Let $V$ be a $\mathbb{Z}$-graded finite-dimensional vector space with a skew-symmetric pairing of degree $1-n$ and
$V=\oplus_i V_i$ is its decomposition by degrees.
Let $M$ be a $n$-dimensional closed oriented manifold 
and $b_i$ its  Betti numbers.
Then
factorization homology $H_*(\int_M \We(V))$ is a one-dimensional 
$k[h^{-1}, h]]$-module of total degree
$$
\sum\limits_{\substack{\{i, j\}\\  i+j  \,\,\mbox{\tiny odd}}} (-i+j)b_i\dim V_j
$$
\label{one}
\end{theorem}

\begin{proof}
Consider the filtration of $\int_M \We_h(V)$ 
by powers of $h$ and the corresponding  spectral sequence.
The associated graded complex calculates the factorization homology of the commutative
polynomial algebra $\We_{h=0}(V)$, the result is given by Proposition \ref{com}.
Let us calculate the 0-th differential.  Since the action of  $\fm$
at the first order by $h$
is given by a differential operator of order $2$, the $0$-th differential
is a differential operator of order $2$ as well.
Thus, it is enough to calculate the differential on the degree $2$
part of algebra $\We_{h=0}$. 

By Proposition \ref{com}, it is equal
to the homology of pairs of points of $M$ labeled by elements of a basis of $V$.
To get the differential of a given homology class one need 
to present it by a cycle, lift this cycle
to the complex, calculating $\int $ and take the differential there.
Present a given class $[c]\in H_*(M^2)$ by a cycle $c$ that intersects the diagonal
$\delta\colon M_\delta \hookrightarrow M^2$ transversally. Then one may see, that differential
of the lifted cycle is $c\cap M_\delta\cdot \omega(x,y)$, where $x$ and $y$
are element of the basis marking the points. In other words, it is equal to
$\delta_* \delta^* [c]\cdot\omega(x,y)$.

Thus, the $0$-th differential is given by the differential operator
of degree two given by the non-degenerate degree 1 pairing on $H_*(M)\otimes V$. 
The resulting complex is the Koszul complex, which has 
the only homology class and consequently the spectral sequence degenerates
at the first term. This only class
is presented by the top degree symmetric power
of the odd part of finite-dimensional vector space $H_*(M)\otimes V$, which gives the formula from the Theorem.
\end{proof}

\begin{example}
Let $n=1$ and $V$ is concentrated in degree $0$. 
Then by Example \ref{moyal}, $\We(V)$
is the usual Weyl algebra. For $M=S^1$ the factorization homology
is the Hochschild homology  and Theorem \ref{one} matches with
the well-known fact about Weyl algebra:
$$
\dim HH_i(\mathcal{W}^1(V))=\begin{cases}
1,& i=\dim V,\\
0, & \mbox{otherwise},
\end{cases} 
$$
see e.~g. \cite{FT}.
\end{example}

The proof of Theorem \ref{one} allows to produce  an explicit cycle
presenting the only non-trivial class in factorization homology
of the Weyl algebra on a closed manifold similarly to the
example. Below we consider the simplest case, leaving the general one to the reader.
 
\begin{prop}
Let $M$ be an odd-dimensional rational homology sphere and the $\mathbb{Z}$-graded vector
space $V$ has only odd-degree components. Then the only non-trivial
cycle in the homology of $\int_M \We(V)$ is presented by a cycle in $C_0(M)$
given by a point marked by an element $S^{\mathrm{top}}V$ of the top
degree in the symmetric power of $V$, since $V$ lies in the odd degree the latter makes sense.
\label{cycle}
\end{prop}

\begin{proof}
This is obviously a cycle and it presents a non-trivial class at the first page
of the spectral sequence from the proof of Theorem \ref{one}.
Since the spectral sequence degenerates at the first page, this cycle survives.
\end{proof}

\subsection{Euler structures}
\label{euler}
As it was mentioned after Proposition \ref{colimit}, framing on a manifold 
simplifies the definition of the factorization complex.
For Weyl $n$-algebra a weaker 
structure is sufficient. 

For a manifold $M$  and a map of
finite sets $S'\to S$ denote by $\CS{M}{S'\to S}$
the fiber product
\begin{equation}
\begin{CD}
  @.  \CS{M}{S'}\\
  @. @VVV      \\
\Cs{M}{S} @>>> {M}^{S'} 
\end{CD}
\end{equation}
where the horizontal map is composition
of the  embedding $\Cs{M}{S}\hookrightarrow M^S$
and the map $M^S\to M^{S'}$
 induced by the map $S'\to S$,
and the vertical map is the projection. 
Space $\CS{M}{S'\to S}$
is equipped with the projection
$$
\pi \colon \CS{M}{S'\to S} \to\Cs{M}{S} .
$$

For the only map from $\set{2}$ to $\set{1}$ the space $\CS{M}{\set{2}\to \set{1}}$
is the total space of the sphere bundle associated with the tangent 
bundle.

\begin{definition}
An {\em Euler structure} on a $n$-manifold $M$
is a closed differential form $\ev$ on $\CS{M}{\set{2}\to \set{1}}$
such that its restriction on any fiber
of the  projection $\CS{M}{\set{2}\to \set{1}}\to M$ is 
the standard volume form on the sphere.
\end{definition}

The only obstruction to the existence of the Euler structure is the
rational Euler class. In particular, on odd-dimensional manifolds
an Euler structure always exists.

Fix an  Euler structure on $M$ given by a form $\ev$ on  $\CS{M}{\set{2}\to \set{1}}$.
For any morphism of arrows from $\set{2}\to \set{1}$ to $S'\to S$
the natural map 
$$
\CS{M}{S'\to S} \to\CS{M}{\set{2}\to \set{1}}
$$
is defined. Denote by $\ev_{ij}$ the pull back of $\ev$ under this map.

Let $V$ be a $\mathbb{Z}$-graded finite-dimensional vector space  equipped with a non-degenerate skew-symmetric pairing 
$\omega\colon V\otimes V\to k$
of degree $1-n$. Let $k[V]$ be the polynomial algebra generated by $V$.
As before let $\ealpha$ be an element of endomorphisms of
$k[V]^{\otimes S} \tens{Aut (S')} C^*(\CS{M}{S'\to S})$ 
given by 
$$
\ealpha= \sum_{\{i,j\}}\partial_\omega^{ij}\wedge\ev_{ij} ,
$$    
where the sum is taken by all 
morphisms of arrows from $\set{2}\to \set{1}$ to $S'\to S$ and
$\partial_\omega^{ij}$ is the operator $\partial_\omega$
applied to the $i$-th and $j$-th factors, where $\partial_\omega$ is defined by (\ref{omega}).
The exponent of $h\ealpha$ in composition with the cup product gives
endomorphism of 
$k[V][[h]]^{\otimes S'} \tens{Aut (S')} C_*(\CS{M}{S'\to S})$.
Consider the composite map 
\begin{equation}
\begin{CD}
k[V][[h]]^{\otimes S'}\tens{Aut (S)}C_*(\CS{M}{S'\to S})\\
@V{\exp(h\ealpha)}VV\\
k[V][[h]]^{\otimes S'}\tens{Aut (S')} C_*(\CS{M}{S'\to S})\\
@V{\mu\otimes {\pi}_*}VV\\
k[V][[h]]^{\otimes S}\tens{Aut (S)} C_*(\Cs{M}{S}),
\end{CD}
\label{composite}
\end{equation}
where $\mu$ is action of morphism in the category $Comm^\otimes$.

\begin{prop} 
\label{EU}
Let $V$ be a $\mathbb{Z}$-graded finite-dimensional vector space   equipped with a non-degenerate 
 skew-symmetric pairing 
$\omega\colon V\otimes V\to k$
of degree $1-n$, $A=\We_h(V)$ be the corresponding Weyl algebra
and $M$ be a closed manifold with an Euler structure.
Then the factorization complex $\int_M \We_h(V)$ 
is the colimit of the diagram
\begin{equation}
\begin{CD}
\bigoplus\limits_{i\colon S'\to S}(C_*(\CS{M}{S'\to S}))\tens{Aut(S')} A^{\otimes S'} \\
@AAA \\
\bigoplus\limits_{S'} C_*(\CS{M}{S'})\tens{Aut(S')} A^{\otimes S'}\\
@VVV \\
\bigoplus\limits_{S}C_*(\Cs{M}{S})\tens{Aut(S)} A^{\otimes S} 
\end{CD}
\label{eu}
\end{equation}
where the downwards arrow is the composite map (\ref{composite})
and the upwards arrow is induced by the natural embedding.
\end{prop}

\begin{proof}
The statement is local along $\Cs{M}{S}$. As the Weyl algebra is invariant (see the remark
before Definition \ref{wna}),  locally it
directly follows from Proposition \ref{colimit} and the definition
of the Weyl algebra.
\end{proof}

\section{Perturbative invariants}

\subsection{Propagator}

Let $M$ be a rational homological sphere of dimension $n$.
Let us denote by  $\tilde M$ the complement in $M$ to the interior of a 
little ball around a point $p\in M$. 

Below we will need the Fulton--MacPherson compactification
of manifolds with boundary. Let $X$ be such a manifold
and  $X\hookrightarrow X'$ be its closed embedding in a manifold of the same
dimension, for example, $X'$ is obtained from $X$ by gluing a collar.
Then  denote by $\CS{\tilde{X}}{S}$ the fiber product
$$
\begin{CD}
  @.  \CS{X'}{S}\\
  @. @VVV      \\
\tilde{X}^S @>>> {X'}^S 
\end{CD}
$$
where the upwards arrow is the embedding and the vertical one is the projection.

Consider the differential $(n-1)$-form
on $\Cs{\R}{\set{2}}$ which is the pullback of the standard form on the sphere
under  the map $(x,y) \mapsto (x-y)/|x-y|$ and continue it on
$\CS{\R}{\set{2}}$ straightforwardly (in Subsection \ref{v} it was denoted by $v$).
Consider the subset of $\CS{\R}{\set{2}}$
where both points lie on the unit
sphere and restrict the form as above to it.
Call the result the {\em standard form}.

The following proposition stays, that on the  2-point Fulton--MacPherson 
configuration space of the ``fake disk'' $\tilde M$ there is a differential $(n-1)$-form
similar to the standard form on the configuration space of the  real disk.

\begin{prop}
For a rational homological sphere $M$ 
choose a point $O$ in the interior of its complement $\tilde{M}$ to
a little disk. Then
on manifold with corners $\CS{\tilde{M}}{\set{2}}$
as above there exists a  differential $(n-1)$-form 
such that
\begin{enumerate}
\item it is smooth and closed;
\item\label{2} its restriction to any fiber of $\pi \colon\CS{\tilde{M}}{\set{2}} \to 
\tilde{M}^2 $
over any point on the diagonal, which is a sphere, 
is equal to 
the standard form on the sphere;
\item\label{3} its restriction to the subset where both points of the configuration
 lie on the boundary equals to the standard form; 
\item\label{4} its restriction to $O\times \partial\tilde{M}$ and $\partial\tilde{M}\times O$ 
equals to the 
standard form on the sphere. 
\end{enumerate}
\label{propagator}
\end{prop}

\begin{proof}
It follows from elementary considerations with Mayer--Vietoris sequence,
see for example \cite{AS}.
\end{proof}

\begin{definition}
We call the  $(n-1)$-form as above on $\CS{\tilde M}{\set{2}}$ a {\em propagator}
and denote it by $\nu$.
\label{Pro}
\end{definition}

Note, that our definition of propagator differs slightly from the one given
in \cite{AS}, \cite{BC}.

\subsection{Collapse}

Let $M$ and
$M'$ be any closed $n$-manifolds.
Choose a point in each manifold and cut off  small open balls
around them. We get two manifolds $\tilde M$ and $\tilde M'$ with boundaries
$S^{n-1}$. Denote their interiors by $M_0$ and $M'_0$.
The connected sum $M\# M'$ is a result of gluing together of
these two manifolds by their boundaries. 
Call the continuous map $\col \colon M\# M'\to M'$ that shrinks $M$
to a point $p \in M'$ by the {\em collapse map}.

In general, the collapse map does not produce any map
between factorization homologies of $M\# M'$ and $M$.
There are two cases when it obviously does.

The first case is when the algebra is commutative.
The factorization homology is given by homology of the
powers of the space and the morphism is given by the
direct image on homology of the powers. 

The second case is when $M=S^n$. Then $M\# M'= M'$.
To build the morphism one need loosely speaking
 to take
everything sitting in $M$, multiply it and put the result to the point
$p$. Arguments as in Proposition \ref{ret} shows that 
this is an isomorphism.

There is  another case when such morphism exists:
when $M$ is an odd-dimensional homology sphere and the algebra at hand
 is the Weyl algebra.
Its construction occupies the rest of this Subsection.

The morphism factorizes through an intermediate object we are going to define.

Let $M$ be a rational homology  odd-dimensional sphere and
$M'$ be any closed $n$-manifold of the same dimension.
Choose Euler structures on $M$ and $M'$, this is possible
because they are odd-dimensional. 
These Euler structures naturally define an Euler structure on the connected sum $M\# M'$ 
due to the following trick, which works for any pair of
odd-dimensional manifolds.
Choose as above small embedded open balls $D\hookrightarrow M$
and   $D\hookrightarrow M'$  and suppose, that the sphere bundle associated with the tangent
bundle is trivialized over $D$ and  the Euler structure is constant there.
To build the connected sum $M\# M'$ one need to glue the complements of
$D$ in $M$ and $M'$ by some orientation-reversing linear automorphism
of the sphere $S=\partial D$. Let us choose the antipodal map.
One may see, that under the natural isomorphism over $S$
of sphere bundles associated with  tangent bundles over $M$ and $M'$,
the Euler structures on $M$ and $M'$  fit together.

For a surjective morphism of manifolds $f\colon X'\to X$
and a map of sets $S'\to S$
define space $\CS{X'/X}{S'\to S}$ as the fiber product
\begin{equation}
\begin{CD}
  @.  \CS{X'}{S'}\\
  @. @VVV      \\
\Cs{X}{S} @>>> {X}^{S'} 
\end{CD}
\end{equation}
where the vertical arrow is the composition
of projection $\CS{X'}{S'} \to {X'}^{S'}$ with $f^{S'}$
and the lower arrow is  composition
of the  embedding $\Cs{M}{S}\hookrightarrow X^S$
and the map $X^S\to X^{S'}$
induced by the map $S'\to S$.
Space $\CS{X'/X}{S'\to S}$
is equipped with the projection
$$
\pi \colon \CS{X'/X}{S'\to S} \to\Cs{X}{S} .
$$

For the collapse map $M\# M'\to M'$ consider space $\CS{M\# M'/M'}{\set{2}\to \set{1}}$.
This space contains $\CS{M'}{\set{2}\to \set{1}}$ and $M_0^2$
as subspaces. On the first one the Euler structure gives a differential 
$(n-1)$-form and on the second one choose a propagator (Definition \ref{Pro}).
Property \ref{3} of  propagator (Proposition \ref{propagator}) allows to glue
it in a global $(n-1)$-cocycle in the cochain complex of  $\CS{M\# M'/M'}{\set{2}\to \set{1}}$.
Denote it by  $\EV$. Note that the space is
not manifold, but $\EV$
is a well-defined cochain of the corresponding relative complex.

Let $V$ be a $\mathbb{Z}$-graded finite-dimensional vector space  equipped with a non-degenerate skew-symmetric pairing 
$\omega\colon V\otimes V\to k$
of degree $1-n$. Let $k[V]$ be the polynomial algebra generated by $V$.
Mimicking construction from Subsection \ref{euler}
let $\Ealpha$ be an element of endomorphisms of
$k[V]^{\otimes S} \tens{Aut (S')} C^*(\CS{M\# M'/M'}{S'\to S})$ 
given by 
$$
\ealpha= \sum_{\{i,j\}}\partial_\omega^{ij}\wedge\EV_{ij} ,
$$    
where the sum is taken by all 
morphisms of arrows from $\set{2}\to \set{1}$ to $S'\to S$ and
$\partial_\omega^{ij}$ is the operator $\partial_\omega$
applied to the $i$-th and $j$-th factors, where $\partial_\omega$ is defined by (\ref{omega}).
The exponent of $h\Ealpha$ in composition with  cup product gives
endomorphism of 
$k[V][[h]]^{\otimes S'} \tens{Aut (S')} C_*(\CS{M\# M'/M'}{S'\to S})$.
Consider the composite map 
\begin{equation}
\begin{CD}
k[V][[h]]^{\otimes S'}\tens{Aut (S)}C_*(\CS{M\# M'/M'}{S'\to S})\\
@V{\exp(h\Ealpha)}VV\\
k[V][[h]]^{\otimes S'}\tens{Aut (S')} C_*(\CS{M\# M'/M'}{S'\to S})\\
@V{\mu\otimes {\pi}_*}VV\\
k[V][[h]]^{\otimes S}\tens{Aut (S)} C_*(\Cs{M'}{S}),
\end{CD}
\label{composite2}
\end{equation}
where $\mu$ is the morphism in the category $Comm^\otimes$.

By analogy with  (\ref{eu}) consider
the diagram
\begin{equation}
\begin{CD}
\bigoplus\limits_{S'} C_*(\CS{M\#M'}{S'})\tens{Aut(S')} A^{\otimes S'}\\
@AAA\\
\bigoplus\limits_{i\colon S'\to S}(C_*(\CS{M\# M'/M'}{S'\to S}))\tens{Aut(S')} A^{\otimes S'}\\
@VVV \\
\bigoplus\limits_{S}C_*(\Cs{M'}{S})\tens{Aut(S)} A^{\otimes S} 
\end{CD}
\label{collapse}
\end{equation} 
where the downwards arrow is the composite map (\ref{composite2})
and the upwards arrow is induced by the natural embedding.

The desired intermediate object is the colimit of diagram (\ref{collapse}). 
Property \ref{2} of propagator (Proposition \ref{propagator}) 
supplies us with  a natural  map  from the diagram presenting 
$\int_{M\# M'} \We_h(V)$ by Proposition \ref{EU} to (\ref{collapse}),
thus with a map from 
$\int_{M\# M'} \We_h(V)$  to the colimit of (\ref{collapse}).  

The following Proposition completes the construction.

\begin{theorem}
The colimit of (\ref{collapse}) is isomorphic to $\int_{M'} \We_h(V)$.
\label{final}
\end{theorem}

\begin{proof}
As it was discussed after Proposition \ref{coend}, 
the factorization complex 
is equipped with 
an increasing filtration by the number of points of the configuration
space of distinct points. Introduce a slightly
different filtration on $\int_{M'}\We_h(V)$
fixing an element of the algebra sitting at the point $p$,
if there nothing at this point, we assume that it is $1$.  
The associated graded object of this filtration 
is $\We_h(V)_p\otimes \bigoplus_{S} C_*(\Cs{M'\setminus 
p}{S})\tens{Aut(S)} {\We_h(V)}^{\otimes S'}$.
 For the same reason, because the horizontal arrow 
in (\ref{collapse}) is surjective, the colimit of (\ref{collapse})
is also filtrated with the same quotients. 

To prove the statement
we are going to  define a map from $\int_{M'} \We_h(V)$
to the colimit of (\ref{collapse}). As it was already discussed 
after Proposition \ref{coend}, every chain in $\int_{M'} \We_h(V)$
may be presented as a chain of $\We_h(V)_p\otimes \bigoplus_{S} C_*(\Cs{M'\setminus 
p}{S})\tens{Aut(S)} {\We_h(V)}^{\otimes S'}$. Take such a representative
$a_0\otimes c\otimes a_1\otimes\cdots\otimes a_s$, where 
$c\in C_*(\Cs{M'\setminus p}{S}$ 
and $a_i\in \We_h(V)$,
and send it to $a_0\otimes a_1\otimes \cdots\otimes a_s \otimes \imath_{p*} c$,
where map 
$$
\imath_p\colon \Cs{S}{M'} \hookrightarrow \CS{M\# M'/M'}{(S\cup p)\stackrel{id}{\to} (S\cup p)}
$$
embeds the configuration and adds the point $p$ to it.

One may see that this map is a map of complexes due to property 
\ref{4}  of the propagator (Proposition \ref{propagator})  and 
an isomorphism on the associated graded object. Consequently, it gives an 
isomorphism of complexes.
\end{proof}

Call the morphism $\Col \colon \int_{M\# M'} \We_h(V) \to \int_{M'} \We_h(V)$ 
just constructed  the {\em collapse morphism}.

The proof of this Proposition may be interpreted by means of
cosheaves in the spirit of the discussion at the end of Subsection \ref{FH}.
Indeed, the colimit of Diagram (\ref{collapse}) gives a cosheaf
on the Ran space of $M'$. Theorem \ref{final} 
states that it is isomorphic to the one given
by the Weyl algebra.
 
Note finally, that Theorem \ref{final} may be reformulated as follows:
for a homological sphere $M$ the factorization complex
$\int_{\tilde{M}} \We_h(V)$ is isomorphic to $\We_h(V)$
as an $\int_{[0,1]\times S^{n-1}} \We_h(V)$-module
(about the module structure
on the factorization complex of a manifold with boundary
see e.~g. \cite{G} and references therein). 

\subsection{Invariants}
\label{Inv}

 Factorization homology of Weyl $n$-algebras may be used to construct 
invariants of manifolds. Let $M$
be a closed $n$-manifold and $V$
be a $\mathbb{Z}$-graded finite-dimensional vector space with a non-degenerate pairing of degree $1-n$.
By Theorem \ref{one}
the factorization homology of $\We_h(V)$
on $M$ is one-dimensional.
The idea of the invariant  we
are going to build is to produce
in some manner a cycle in $\int_M \We_h(V)$ 
 and calculate 
the class represented by it. As the homology group is one-dimensional,
this class is a multiple of a standard one.
The series we get this way is the invariant of the manifold.

Let us restrict ourselves with the following conditions:  
$M$ is a rational homology sphere
of odd dimension $n$ and $V$ be a $\mathbb{Z}$-graded finite-dimensional vector space,
which has only odd-dimensional components.
Under these conditions due to Proposition \ref{cycle}
the only class in the factorization homology
is presented by an especially simple cycle,
just an element of the top degree power of $V$ sitting at a point,
call this cycle the standard one.

To produce a different cycle we shall resort to the morphism 
given by Proposition \ref{modules}. It sends the Chevalley--Eilenberg complex
of the Lie algebra $\LL(\We(V))$
associated with the Weyl algebra  $\We(V)$ 
to the factorization complex of $\We(V)$.

By Proposition, for $n>1$ \ref{poiss} $\LL(\We(V))$ 
is $\mathbb{Z}$-graded Poisson Lie algebra. Suppose that $\dim V \ge 3$ and denote by $\LL(\We(V)^{\ge 3})$
the Lie subalgebra of polynomials of degree not less than 3.
One may see that a generator of $S^{top}V$ is in the center of $\LL(\We(V)^{\ge 3})$.
Thus the map 
\begin{equation*}
k\to \LL(\We(V)^{\ge 3}),
\label{top}
\end{equation*}
which sends the generator to a non-zero element from $S^{top}V$ is a morphism from the trivial 
$\LL(\We(V)^{\ge 3})$-module to the adjoint one. Consider the induced map
$$
C_{Ch}(\LL(\We(V)^{\ge 3})) \to C_{Ch}^{ad} (\LL(\We(V)^{\ge 3}))
$$
and combine it with map 
$$
 C_{Ch}^{ad} (\LL(\We(V)^{\ge 3}))\to \int _M \We(V)
$$
given by Proposition \ref{modules}.
The composite map
\begin{equation}
 C_{Ch} (\LL(\We(V)^{\ge 3}))\to \int _M \We(V)\xrightarrow{\sim} k[[h^{-1}, h]]
\label{invariant}
\end{equation}
is the desired invariant. In other words,
the invariant is a cohomology class of total degree zero of
$\mathbb{Z}$-graded Lie algebra $\LL(\We(V)^{\ge 3})$ with coefficients in $k[[h^{-1}, h]]$.  
To get just  an element of $k[[h^{-1}, h]]$ 
one may substitute a homology class of this Lie algebra in it.
Note, that the coefficients of this series are rational due to the remark 
preceding Example \ref{moyal}.

As it was already mentioned in the Introduction,
a  cocycle of the complex linear dual to the factorization
complex  $\int _M \We(V)$ which does 
not vanish on the standard cycle would make this invariant more explicit.
As such a cocycle is unavailable, we shall make use of the collapse morphism
from the previous Subsection. 

\begin{prop}
If $M$ and $M'$ are both rational homology odd-dimensional spheres and $V$ has only odd-degree 
components
then the collapse morphism 
 $$\Col \colon \int_{M\# M'} \We(V) \to \int_{M'} \We(V)$$ 
induces isomorphism on homologies.
\end{prop}

\begin{proof}
By Proposition \ref{cycle}, the non-trivial class in homology of  
$\int_{M\# M'} \We(V)$ 
is presented by a cycle which is an element of the algebra sitting at a point.
As it follows from its definition, the collapse morphism 
sends it to the same cycle in $\int_{M'} \We(V)$, which is
non-trivial by Proposition \ref{cycle}.    
\end{proof}

Assuming $M'=S^n$ in the Proposition above
we get an isomorphism $\int_{M} \We(V)\to \int_{S^n}\We(V)$.
In composition with (\ref{invariant}) we get a morphism
\begin{equation*}
 C_{Ch} (\LL(\We(V)^{\ge 3}))\to \int _{S^n} \We(V),
\end{equation*}
which is better than (\ref{invariant}), because the target does not depend on $M$.

Unwinding the definition of the collapse morphism
one may see that this cocycle of $\LL(\We(V)^{\ge 3})$ 
taking values in $\int_{S^n}\We(V)$ is a sort of cocycle
given by the graph complex, see
\cite{ Kon2, Kon1, 
Zab}. It is known (\cite{AS, AS2}), that perturbative Chern--Simons invariants
also give classes in the graph complex in the same way, by integration of the
powers of the propagator. It makes us believe that our invariants
coincide with the perturbative Chern--Simons ones. 
Perhaps, some good choice of the propagator will lead
to a more explicit formula.

Finally, let $n=3$, $V$ be a $\mathbb{Z}$-graded finite-dimensional vector space of dimension more than 
$2$ concentrated in degree
$1$ with skew-symmetric pairing of degree $-2$, that is $V[1]$
is equipped with a symmetric pairing.
In this case for dimensional reasons the cocycle is given by trivalent 
graphs. If $V[1]$ is the underlying space of a Lie algebra
with non-degenerate pairing, then the element in $S^3 V[1]$,
which is the composition of the Lie bracket and the pairing,
is a Maurer--Cartan 
element in $\LL(\We(V)^{\ge 3})$. Its power gives a homology class.
Values of the cocycle on it must be the perturbative invariants 
associated with given Lie algebra. More about this case
the reader may find in the Appendix.

\section*{Appendix}

The physical definition of perturbative Chern--Simons invariant
is based on the asymptotic  series of the oscillating integral
 $\int e^{iS}$
taken over the space of all $G$-connection $A$ on $M$, where 
$S=\frac{\kappa}{4\pi}\int_M \mathop\mathrm{tr}(A\wedge dA+\frac{2}{3}A\wedge A\wedge A)$ is the 
Chern--Simons 
functional, $M$ is a $3$-manifold and $G$ is a semi-simple Lie group. 
The aim of this appendix is to demonstrate speculatively
how to interpret the calculation of such an integral in terms of the factorization complex.

Thus, we have the infinite-dimensional space of connections, a function $S$
on it and we want to calculate the asymptotic series in $1/\kappa$ of the oscillating integral.
If $M$ is a homology sphere, then function $S$ has a non-degenerate 
critical point at the origin. Thus, the free term of the series in hand
is the Gaussian integral by an infinite-dimensional space
and is unapproachable by algebraic methods. But after dividing by this term
the series may be calculated by means  of the method of Feynman diagrams.

To explain how this method works consider an abstract situation,
the reader can find more at \cite{JF}. 
Let $V$ be a Euclidean vector space and
$f$ be a smooth function on it 
such that its Taylor series at the origin start with terms of degree, at least, three.
Choose a volume form on $V$ and consider the integral $\int e^{(-{\lvert x\rvert}^2+tf)}$.
Consider the twisted de Rham complex of polynomial forms 
$\Omega^*_t$ given by differential forms on $V$
with differential $d_{dR}-2(\mathbf{x}, d\mathbf{x})+ t\,df$,
where $d_{dR}$ is the de Rham differential.
One may see that complex $\Omega^*_t\otimes \mathbb{R}[[t]]$
has only top degree cohomology, which is one-dimensional over $\mathbb{R}[[t]]$.
This one-dimensional vector bundle over the $t$-line 
has the Gau\ss--Manin connection and a section given by 
the chosen volume form on $V$. Their quotient
is a series in $t$ up to a constant factor
and one may show that this is the asymptotic expansion of the oscillating integral
up to a constant.

We are now going to construct a $\mathfrak{fm}_3$-algebra 
(or equivalently, by Proposition \ref{eq}, an $e_3$-algebra) the factorization
complex of which on a homology $3$-sphere
$M$ resembles the twisted de Rham complex as above.
Let $g$ be a Lie algebra with a non-degenerate invariant bilinear form.
The desired  $\mathfrak{fm}_3$-algebra is a deformation of the Chevalley--Eilenberg 
commutative
$dg$-algebra $C_{Ch}^\bullet(g)$ in the class of $\mathfrak{fm}_3$-algebras. The deformation
may be described as follows: forget about the differential on the Chevalley--Eilenberg
complex and deform the underlying polynomial algebra as in the
definition of the Weyl algebra, that is  
apply Definition  \ref{wna} 
to the space $g^\vee$ and the
pairing given by the invariant bilinear form.
It is easy to check that this deformation respects the differential.
Note, that this $e_3$-algebra is the algebra of $\mathrm{Ext}$'s 
from the unit to itself in 
$e_2$-category of representations of a quantum group.
Denote it by $Ch^\bullet_h(g)$.

Alternatively, this  $\mathfrak{fm}_3$-algebra may be defined as follows. Start
with  $\mathbb{Z}$-graded finite-dimensional vector space $g^\vee[1]$ with the pairing of degree $-2$
given by the invariant scalar product and build the Weyl algebra
$\mathcal{W}^3_h(g^\vee[1])$. Then define differential on it
as $\frac{1}{h}\{\cdot,q\}$, where $\{,\}$ is image of the Lie bracket under (\ref{morphism})
and $q\in S^3(g^\vee[1])$ is the composition of the Lie  bracket
on $g$ and the scalar product. 
One may show, that similarly to Hochschild homology (see e.~g. \cite[Propositon 1.3.3]{Lo}),
factorization homology on a closed manifold is invariant under inner deformations.
It follows 
by Theorem \ref{one} 
that the homology of $\int_M Ch^\bullet_h\otimes k[h^{-1}, h]]$
is free $k[h^{-1}, h]]$-module of rank 1. And moreover,
this homology   is equipped with
a connection along the formal deleted $h$-line.

To fulfill the analogy (note, that
$t$ corresponds to $1/h$) we have to present a section
of this one-dimensional vector bundle and compare it with
a horizontal one. 
Formula (\ref{invariant}) produces elements in the factorization complex
of $Ch_h^\bullet(g)$. One may see, that $1$ goes to a cycle
(in fact, this is the cycle given by Proposition \ref{cycle})
and this is an analog
of the section given by the volume form  on $V$ in the example above.
On the other hand, one may see that a cycle horizontal  with respect to the connection
is the image under  (\ref{cycle}) of the cycle
$
\sum_i\frac{h^{-i}}{i!}\underbrace{q\wedge\cdots\wedge q}_i.
$
The quotient of these two sections  is an analog of the asymptotic series
and is given by the invariants as in Subsection \ref{Inv}.

\bibliographystyle{alphanum}
\bibliography{weyl}

\end{document}